\documentclass[11pt,a4paper]{amsart} 
\usepackage[margin=3cm]{geometry}
\usepackage[T1]{fontenc}

\usepackage{latexsym, amssymb, amsmath}
\usepackage{enumitem}
\usepackage{amsthm}
\usepackage{tikz-cd}
\usepackage{mathtools}
\usepackage{orcidlink}
\usepackage{doi}

\usepackage[utf8]{inputenc}
\usepackage[british]{babel}
\usepackage[all]{xy}
\usepackage{hyperref}
\usepackage{calrsfs}
\usepackage{mathabx}
\usepackage{xcolor}
\usepackage{mathtools}
\usepackage{blkarray}
\usepackage[textsize=scriptsize]{todonotes}
\usepackage{thmtools,mathtools}
\usepackage{nameref,hyperref,cleveref}
\usepackage{quiver}
\usepackage{verbatim}
\usepackage{float}
\usepackage{tablefootnote}
\usepackage{refcount}
\usepackage{footmisc}
\usepackage{microtype}

\theoremstyle{plain}
\newtheorem{theorem}[subsection]{Theorem}
\newtheorem{proposition}[subsection]{Proposition}

\theoremstyle{definition}
\newtheorem{definition}[subsection]{Definition}

\theoremstyle{remark}
\newtheorem{remark}[subsection]{Remark}
\newtheorem{example}[subsection]{Example}

\numberwithin{equation}{section}

\newenvironment{tfae}
{
\begin{enumerate}}
{\end{enumerate}}

\newcommand{\noproof}{\hfil\qed}

\newcommand{\bC}{\mathsf{C}}
\newcommand{\bF}{\mathbb{F}}

\DeclareMathOperator{\bchar}{char}

\DeclareMathOperator{\Aut}{Aut}
\DeclareMathOperator{\op}{op}
\DeclareMathOperator{\Act}{Act}
\DeclareMathOperator{\Hom}{Hom}

\DeclareMathOperator{\SplExt}{SplExt}
\DeclareMathOperator{\Imm}{Im}

\DeclareMathOperator{\Der}{Der}
\DeclareMathOperator{\DerL}{Der_{Lie}}
\DeclareMathOperator{\ADer}{ADer}
\DeclareMathOperator{\Bider}{Bider}
\DeclareMathOperator{\Inn}{Inn}
\DeclareMathOperator{\ad}{ad}
\DeclareMathOperator{\Ad}{Ad}
\DeclareMathOperator{\Leib}{Leib}
\DeclareMathOperator{\Zl}{Z_l}
\DeclareMathOperator{\Zr}{Z_r}
\DeclareMathOperator{\Z}{Z}
\DeclareMathOperator{\ZL}{Z_{Lie}}
\DeclareMathOperator{\hol}{Hol}
\DeclareMathOperator{\holL}{hol_{Lie}}

\DeclareMathOperator{\USGA}{USGA}

\newcommand{\Lie}{\ensuremath{\mathsf{Lie}}}
\newcommand{\sLeib}{\ensuremath{\mathsf{Leib}}}

\newcommand{\Grp}{\ensuremath{\mathsf{Grp}}}

\newdir{>}{{}*:(1,-.2)@^{>}*:(1,+.2)@_{>}}
\newdir{<}{{}*:(1,+.2)@^{<}*:(1,-.2)@_{<}}
\newdir{>>}{{}*!/3.5pt/:(1,-.2)@^{>}*!/3.5pt/:(1,+.2)@_{>}*!/7pt/:(1,-.2)@^{>}*!/7pt/:(1,+.2)@_{>}}
\newdir{ >}{{}*!/-8pt/@{>}}

\newcommand{\ov}[1]{\overline{#1}}

\date{}

\usepackage{csquotes}

\usepackage[%
backend=bibtex,%
style=numeric-comp,%
sorting=nyt,%
hyperref,%
maxcitenames=4, mincitenames=4,
maxbibnames=99, minbibnames=99 %
]{biblatex}
\AtBeginBibliography{}

\DeclareFieldFormat{pages}{#1}

\renewbibmacro{in:}{%
\ifentrytype{article}{}
{\bibstring{in}
\printunit{\intitlepunct}}}

\begin{document}

\title[On Lie-holomorphs of Leibniz algebras]{On Lie-holomorphs of Leibniz algebras}

\author[G.~La Rosa]{Gianmarco La Rosa~\orcidlink{0000-0003-1047-5993}}
\author[M.~Mancini]{Manuel Mancini~\orcidlink{0000-0003-2142-6193}}

\email{gianmarco.larosa@unipa.it}
\email{manuel.mancini@uclouvain.be; manuel.mancini@unipa.it}

\address[G.~La Rosa, M.~Mancini]{Dipartimento di Matematica e Informatica, Università degli Studi di Palermo, via Archirafi 34, 90123 Palermo, Italy}

\address[M.~Mancini]{Institut de Recherche en Mathématique et Physique, Université catholique de Louvain, chemin du cyclotron 2 bte L7.01.02, B--1348 Louvain-la-Neuve, Belgium.}


\begin{abstract}
We study the notion of the \emph{Lie-holomorph} of a Leibniz algebra, recently introduced by N.~P.~Souris as a generalisation of the classical holomorph construction for Lie algebras. We establish a connection between the Lie-holomorph construction and the Leibniz algebra of biderivations defined by J.-L.~Loday, and we prove that a linear endomorphism is a Lie-derivation if and only if it is simultaneously a derivation and an anti-derivation. As an application, we classify the Lie-holomorph algebras of all low-dimensional non-Lie Leibniz algebras over a field of characteristic different from $2$.
\end{abstract}


\maketitle

\section{Introduction}
In 1965, A.~Blokh introduced algebraic structures under the name of $D$-algebras; they are now known as Leibniz algebras~\cite{blokhLie1965}. Later, in 1993, they were popularised by J.-L.~Loday~\cite{loday1993}, who presented them as a non-anticommutative version of Lie algebras. A right Leibniz algebra $L$ is a vector space equipped with a bilinear product $[-,-]$, usually called \emph{bracket} or \emph{commutator}, such that the right multiplications act as derivations. With this definition, Lie algebras appear as a special class of Leibniz algebras.

The notion of the holomorph of a Lie algebra is well known in Lie theory
(see, for instance,~\cite{jacobsonLie}). More precisely, given a Lie algebra $L$, the vector space $L \times \Der(L)$, where $\Der(L)$ denotes the Lie algebra of derivations of $L$, can be endowed with a Lie algebra structure via the bilinear multiplication
\[
[(x,d),(y,d')]=\left([x,y] + d(y) - d'(x),\,d \circ d' - d' \circ d \right).
\]
The resulting Lie algebra $(L \times \Der(L),[-,-])$ is called the \emph{holomorph}
of $L$, and it is usually denoted by $\hol(L)$. The study of this construction
goes back at least to the end of the nineteenth century and originates in group
theory, where the holomorph encodes, in a single object, a group together with its
automorphism group (see, e.g.,~\cite{Burnside_2012}).

In the context of Leibniz algebras, the notion of holomorph was first introduced by K.~Boyle, K.~C.~Misra, and E.~Stitzinger in~\cite{misra}. With their definition, the authors were able to extend several results from complete Lie algebras to the setting of complete Leibniz algebras. However, in this framework the Leibniz algebra $L$ is not embedded as an ideal in its holomorph, and when $L$ is a Lie algebra, the classical holomorph construction is not recovered.

For these reasons, in this manuscript we adopt the definition of \emph{Lie-holomorph} given by N.~P.~Souris in~\cite{souris}. It turns out that the study of the Lie-holomorph of a Leibniz algebra~$L$ is closely related to the Leibniz algebra of biderivations $\Bider(L)$, which was introduced by J.-L.~Loday in~\cite{loday1993} in terms of pairs consisting of derivations and antiderivations of $L$ satisfying a compatibility condition. This connection represents the core of the present work.

After recalling some basic definitions and preliminaries in \Cref{sec:preliminaries}, we present the main results concerning biderivations and Lie-holomorphs of Leibniz algebras in \Cref{sec:holomorph}. Eventually, in \Cref{sec:classification} we provide a complete classification of the Lie-holomorphs of all low-dimensional non-Lie Leibniz algebras, by applying the results developed here together with the classification of Leibniz algebras of dimension at most four given in~\cite{4-dim-nil, 4-dim-Misra, 4-dim-solv}, and the classification of biderivations given in~\cite{ManciniBider}.

\section{Preliminaries}\label{sec:preliminaries}
We assume that $\bF$ is a field with $\bchar(\bF)\neq2$. For the general theory we refer to~\cite{ayupov2019,loday1993}.

\begin{definition}\cite{loday1993}
A \emph{right Leibniz algebra} over $\bF$ is a vector space $L$ over $\bF$ endowed with a bilinear map (called \emph{bracket} or \emph{commutator}) $\left[-,-\right]\colon L\times L \to L$ which satisfies the \emph{right Leibniz identity}
\[
\left[\left[x,y\right],z\right]=\left[[x,z],y\right]+\left[x,\left[y,z\right]\right],
\]
for any $x,y,z \in L$.
\end{definition}

In the same way we can define a left Leibniz algebra, using the left Leibniz identity
\[
\left[x,\left[y,z\right]\right]=\left[\left[x,y\right],z\right]+\left[y,\left[x,z\right]\right].
\]

A Leibniz algebra that is both left and right is called \emph{symmetric} Leibniz algebra.

\begin{remark}
Given a right (resp.~left) Leibniz algebra $L$, the multiplication
\[
[x,y]^{\op}=[y,x]
\]
defines a left (resp.~right) Leibniz algebra structure on the same underlying vector space. The resulting Leibniz algebra is denoted by $L^{\op}$. Hence, there is an equivalence between left and right Leibniz algebras.

From now on, throughout the rest of the paper, we work with right Leibniz algebras.
\end{remark}
Every Lie algebra canonically defines a Leibniz algebra, and a Leibniz algebra whose bracket is skew-symmetric necessarily satisfies the Jacobi identity, hence is a Lie algebra. This correspondence gives rise to an adjunction (see~\cite{maclane2013}) between the category $\Lie$ of Lie algebras over $\bF$ and the category $\sLeib$ of Leibniz algebras over $\bF$. More precisely, the forgetful functor $i\colon \Lie \to \sLeib$ admits a left adjoint
\[
p \colon \sLeib \to \Lie,
\]
which assigns to each Leibniz algebra $L$ the quotient $L/\Leib(L)$. Here
\[
\Leib(L)=\langle [x,x] \mid x\in L \rangle
\]
denotes the \emph{Leibniz kernel} of $L$, which is the minimal two-sided ideal such that the quotient $L/\Leib(L)$ becomes a Lie algebra.

The left and the right center of a Leibniz algebra $L$ are, respectively,
\[
\Zl(L)=\left\{x\in L \mid \left[x,L\right]=0\right\} \; \text{and} \; \Zr(L)=\left\{x\in L \mid \left[L,x\right]=0\right\},
\]
while the \emph{center} of $L$ is
\[
\Z(L)=\Zl(L)\cap \Zr(L).
\]
One may check that $\Zr(L)$ and $\Z(L)$ are ideals of $L$, whereas the left center is not even a subalgebra. Moreover, $\Leib(L) \subseteq \Zr(L)$.

In~\cite{casas_center}, the authors introduced the notion of \emph{Lie-center} of a Leibniz algebra $L$
\[
\ZL(L)=\lbrace x \in L \mid [x,y]+[y,x]=0, \, \forall y \in L \rbrace.
\]
One may check that $\ZL(L)$ is a characteristic ideal of $L$ containing $Z(L)$. Furthermore, $\ZL(L)=L$ if and only if $L$ is a Lie algebra.

We further recall the definitions of nilpotent and solvable Leibniz algebras.

\begin{definition}
Let $L$ be a Leibniz algebra and let
\[
L^{(0)}=L, \quad L^{(k+1)}=[L^{(k)},L], \quad \forall k \geq 0
\]
be the \emph{lower central series} of $L$. $L$ is \emph{$n$-nilpotent} if $L^{(n-1)} \neq 0$ and $L^{(n)}=0$.
\end{definition}


\begin{definition}
Let $L$ be a Leibniz algebra and let
\[
L^{0}=L, \quad L^{k+1}=[L^{k},L^k], \quad \forall k \geq 0
\]
be the \emph{derived series} of $L$. $L$ is \emph{$n$-solvable} if $L^{n-1} \ne 0$ and $L^{n}=0$.
\end{definition}

\begin{remark}
If $L$ is $2$-nilpotent, i.e., if $[x,[y,z]]=[[x,y],z]=0$ for any $x,y,z \in L$, then the derived subalgebra $[L,L]$ lies in the Lie-center $\ZL(L)$. 
\end{remark}

The definition of a derivation for a Leibniz algebra is the same as in the case
of Lie algebras, and in fact it coincides with the standard notion of derivation
in the broader setting of non-associative algebras.

\begin{definition}
Let $L$ be a Leibniz algebra. A \emph{derivation} of $L$ is a linear map $d \colon L \to L$ such that 
\[
d(\left[x,y\right])=\left[d(x),y\right]+\left[x,d(y)\right],
\]
for any $x,y \in L$.
\end{definition}

The right multiplications are particular derivations called \emph{inner derivations}, and an equivalent way to define a Leibniz algebra is to say that the adjoint map $\ad_x=\left[-,x\right]$ is a derivation, for every $x \in L$.

The set $\Der(L)$ is a Lie algebra with the natural bracket given by the commutator $d_1\circ d_2- d_2\circ d_1$. The subset $\Inn(L)$ of all inner derivations of $L$ is an ideal of $\Der(L)$.

The definitions of \emph{anti-derivation} and \emph{biderivation} for a Leibniz algebra were first given by J.-L.~Loday in~\cite{loday1993}.

\begin{definition}\cite{loday1993}
Let $L$ be a Leibniz algebra. An \emph{anti-derivation} of $L$ is a linear map $D \colon L \to L$ such that 
\[
D([x,y])=[D(x),y]-[D(y),x],
\]
for any $x,y \in L$.
\end{definition}

For a left Leibniz algebra, one needs to ask that
\[
D([x,y])=[x,D(y)]-[y,D(x)],
\]
for any $x,y \in L$.

We observe that in the case of Lie algebras, there is no difference between a derivation and an anti-derivation. Moreover, it is easy to check that, for any $x \in L$, the left adjoint map $\Ad_x=[x,-]$ is an anti-derivation.

\begin{remark}\cite{loday1993}
The set $\ADer(L)$ of anti-derivations of a Leibniz algebra $L$ has a $\Der(L)$-module structure with the action
\[
d \cdot D = D \circ d - d \circ D,
\]
for any $d \in \Der(L)$ and for any $D \in \ADer(L)$.
\end{remark}

\begin{remark}
Let $D \colon L \to L$ be an anti-derivation. Then, $D(\Leib(L))=0$ since
\[
D([x,x])=[D(x),x]-[D(x),x]=0,
\]
for any $x \in L$.
\end{remark}

\begin{definition}\cite{loday1993}
Let $L$ be a Leibniz algebra. A \emph{biderivation} of $L$ is a pair
\[
(d,D) \in \Der(L) \times \ADer(L)
\]
such that
\begin{equation}\label{cond}
[x,d(y)]=[x,D(y)],    
\end{equation}
for any $x,y \in L$.
\end{definition}

The set of all biderivations of $L$, denoted by $\Bider(L)$, has a Leibniz algebra structure with the bracket
\[
[(d,D),(d',D')]=(d \circ d' - d' \circ d, D \circ d' - d' \circ D),
\]
for any $(d,D),(d',D') \in \Bider(L)$, and it is possible to define a Leibniz algebra morphism
\[
L \to \Bider(L) \colon x \mapsto (-\ad_x, \Ad_x).
\]
The pair $(-\ad_x, \Ad_x)$ is called \emph{inner biderivation} of $L$, and the set of all inner biderivations forms a Leibniz subalgebra of $\Bider(L)$.

\begin{remark}
Condition~\eqref{cond} is equivalent to requiring that $\Imm(d-D)\subseteq \Zr(L)$. Therefore, it is defined a linear map
\[
\Bider(L) \to \Hom(L,\Zr(L)) \colon (d,D) \mapsto d-D,
\]
where $\Hom(L,\Zr(L))$ denotes the vector space of linear maps from $L$ to $\Zr(L)$.
\end{remark}

\begin{remark}
Biderivations of Leibniz algebras have been studied, among others, in~\cite{CigoliManciniMetere}, where it was proved that the variety $\sLeib$ is \emph{weakly action representable} (see~\cite{WRA,WRAAlg,XabiMancini}), with a weak representing object of a Leibniz algebra $X$ being the biderivation algebra $\Bider(X)$. This means that, for any Leibniz algebra $X$, there exists a natural monomorphism of functors
\[
\tau \colon \SplExt(-,X) \rightarrowtail \Hom(-,\Bider(X)),
\]
where $\SplExt(-,X)$ denotes the functor that maps every object $B$ of $\sLeib$ to the set of isomorphism classes of split extensions of $B$ by $X$. The component 
\[
\tau_B \colon \SplExt(B,X) \rightarrowtail \Hom(B,\Bider(X))
\]
sends any split extension
\[
\begin{tikzcd}
{0} & {X} & {A} & {B} & {0}
\arrow[from=1-1, to=1-2]
\arrow["{k}", tail, from=1-2, to=1-3]
\arrow["{p}", twoheadrightarrow, shift left, from=1-3, to=1-4]
\arrow["{s}", tail, shift left, from=1-4, to=1-3]
\arrow[from=1-4, to=1-5]
\end{tikzcd}
\]
to the Leibniz algebra homomorphism
\[
B \to \Bider(X) \colon b \mapsto (-r_b,l_b),
\]
called \emph{acting morphism}, where the bilinear maps
\[
l \colon B \times X \to X, \quad r \colon X \times B \to X
\]
defined by $l_b(x)=l(b,x)=[s(b),k(x)]$ and $r_b(x)=r(x,b)=[k(x),s(b)]$, denote the \emph{derived action} of $B$ on $X$ (see~\cite{casas, orzech} for more details). Furthermore, one may check that
\[
l_b \circ (l_{b'} - r_{b'})=0,
\]
for any $b,b' \in B$.
\end{remark}

Before concluding this section, we want to recall Remark 2.5 of~\cite{dibartolo2025biderivationscompleteleibnizalgebras}.

\begin{remark}
In the literature on Leibniz algebras, there coexist two different notions of biderivation. Following J.-L.~Loday, biderivations are defined as pairs consisting of a derivation and an antiderivation satisfying a compatibility condition, while more recent works treat biderivations as bilinear maps satisfying Leibniz-type identities in each argument.
	
These definitions are not equivalent, and the objects they produce differ substantially. In this paper, we adopt the former perspective, and throughout we refer to biderivations in the sense of J.-L.~Loday.
\end{remark}

\section{The Lie-holomorph of a Leibniz algebra}\label{sec:holomorph}

In this section we recall the definition of \emph{Lie-holomorph} of a Leibniz algebra, which was recently introduced in~\cite{souris} in order to generalise the holomorph construction for Lie algebras.

Given any Lie algebra $L$, the holomorph $\hol(L)$ is defined as the semidirect product $L \rtimes \Der(L)$ endowed with the bracket
\[
[(x,d),(y,d')]=([x,y]+d(y)-d'(x),[d,d']),
\]
for any $(x,d),(y,d') \in L \rtimes \Der(L)$. One may check that $\hol(L)$ is a Lie algebra, and $L$ embeds as an ideal of $\hol(L)$ via the canonical inclusion
\[
i_1 \colon L \rightarrowtail \hol(L) \colon x \mapsto (x,0).
\]
This construction has analogies with the notion of holomorph in the category $\Grp$ of groups, where the holomorph of a group $G$ is defined as $\hol(G)=G \rtimes \Aut(G)$, with multiplication
\[
(g,\varphi) \cdot (h,\varphi')=(g \cdot \varphi(h),\varphi \circ \varphi').
\]
In fact, both the categories $\Lie$ and $\Grp$ are action representable~\cite{BJK2}: this means that, for any object $X$, there exists an object $\Act(X)$, called the \emph{actor} of~$X$, together with a natural isomorphism of functors
\[
\SplExt(-,X) \cong \Hom(-,\Act(X)).
\]
In the case of Lie algebras, the actor of a Lie algebra $L$ is $\Der(L)$; while for groups, the actor of a group~$G$ is the automorphism group $\Aut(G)$. Hence, the holomorph construction of an object~$X$ of an action representable category is then precisely the semidirect product $\hol(X)=X \rtimes \Act(X)$.

However, the notion of action representable category has been proved to be quite restrictive. For instance, it was proved in~\cite{tim} that the only action representable varieties of non-associative algebras over an infinite field $\bF$, with $\bchar(\bF) \neq 2$, are the varieties of abelian algebras and Lie algebras.

On the other hand, the article~\cite{casas} showed that the weaker notion of \emph{universal strict general actor} (USGA, for short) is available for any Orzech category of interest~\cite{orzech}. For instance, in the category $\sLeib$, $\USGA(L)=\Bider(L)$ for any Leibniz algebra $L$.

It was also proved in~\cite[Theorem 3.9]{casas} that an object $X$ of a category of interest $\bC$ has an actor $\Act(X)$ if and only if the semidirect product $X \rtimes \USGA(X)$ is an object of $\bC$. If this is the case, then $\USGA(X)=\Act(X)$.

As a consequence, since $\sLeib$ is not action representable, the semidirect product $L \rtimes \Bider(L)$ does not  give in general a Leibniz algebra\footnote{It is proved in~\cite[Section 5]{casas} that $L \rtimes \Bider(L)$ is a Leibniz algebra whenever $\Z(L)=0$ or $[L,L]=L$.}. Hence, the holomorph construction $L \rtimes \USGA(L)$ does not extend from Lie algebras to Leibniz algebras. However, two different notions of holomorph in the variety $\sLeib$ are available in the literature.

The authors of~\cite{misra} defined the holomorph of a left Leibniz algebra $L$ as the vector space $L \times \Der(L)$ endowed with the bilinear multiplication
\[
[(x,d),(y,d')]=([x,y]+d(y),[\Ad_x,d']+[d,d'])
\]
for any $(x,d),(y,d') \in L \times \Der(L)$.

This notion has the advantage that certain results for complete Lie algebras extend naturally to complete Leibniz algebras. On the other hand, one may check that $L$ does not embed as an ideal of its holomorph. Moreover, if $L$ is a Lie algebra, the multiplication above does not define a Lie algebra structure on $L \times \Der(L)$, so the construction does not recover the classical holomorph $\hol(L)=L \rtimes \Der(L)$.

In order to solve this issue, N.~P.~Souris considered in~\cite{souris} the ideal
\[
\DerL(L)=\lbrace d \in \Der(L) \mid \Imm(d) \subseteq \ZL(L) \rbrace
\]
of \emph{Lie-derivations} of $L$, and defined the \emph{Lie-holomorph} of a Leibniz algebra $L$ as follows.

\begin{definition}\cite{souris}
Let $L$ be a Leibniz algebra. The \emph{Lie-holomorph} of $L$ is the semidirect product
\[
\holL(L)=L \rtimes \DerL(L)
\]
endowed with the bracket
\[
[(x,d),(y,d')]=([x,y]+d(y)-d'(x),[d,d']),
\]
for any $(x,d),(y,d') \in L \rtimes \Der(L)$.
\end{definition}

He also proved that $\holL(L)$ is a Leibniz algebra, and that $L$ can be embedded as an ideal into $\holL(L)$ via the canonical map $i_1 \colon x \mapsto (x,0)$. Furthermore, if $L$ is a Lie algebra, then $\DerL(L)=\Der(L)$, but the converse does not hold in general as the following example shows.

\begin{example}
Let $L$ be the $3$-dimensional Leibniz algebra with basis $\lbrace e_1,e_2,e_3 \rbrace$ and non-zero commutators $[e_3,e_2]=-[e_2,e_3]=e_1$ and $[e_3,e_3]=-e_1$. One has that $\Leib(L)=\langle e_1 \rangle$, hence $L$ is not a Lie algebra. We want to show that $\DerL(L)=\Der(L)$.
	
We firstly observe that $e_1 \in Z(L)$. Moreover, $x=x_2 e_2 + x_3 e_3 \in \ZL(L)$ if and only if~$x_3=0$. Thus, $\ZL(L)=\langle e_1,e_2 \rangle$.
	
From the classification of biderivations of low-dimensional Leibniz algebras given in~\cite{ManciniBider}, it turns out that $\Der(L)$ can be represented as the Lie subalgebra
\[
\left\lbrace
\begin{pmatrix}
0 & 0 & b \\
0 & a & c \\
0 & 0 & 0 \\
\end{pmatrix} \; \middle| \; a,b,c \in \bF \right\rbrace
\]
of $\mathfrak{gl}(3,\bF)$. Thus, $\Imm(d) \subseteq \ZL(L)$ for any $d \in \Der(L)$, and $\DerL(L)=\Der(L)$.
\end{example}

One may check that the Leibniz algebra $L$ considered above is symmetric. This is not a coincidence, as the following result shows.

\begin{proposition}\label{prop_sym}
Let $L$ be a Leibniz algebra. If $\DerL(L)=\Der(L)$, then $L$ is symmetric.
\end{proposition}

\begin{proof}
If $\DerL(L)=\Der(L)$, then the left adjoint map $\Ad_x=[x,-]$ is a derivation, for any $x \in L$. Indeed, one has
\[
[x,[y,z]]=[[x,y],z]-[[x,z],y],
\]
for any $y,z \in L$, since $\Ad_x \in \ADer(L)$, and
\[
[[x,z],y]=-[y,[x,z]]
\]
since $\ad_z=[-,z] \in \DerL(L)$. Hence, the left Leibniz identity
\[
[x,[y,z]]=[[x,y],z]+[y,[x,z]]
\]
holds in $L$.
\end{proof}

The following example shows that the converse of \Cref{prop_sym} does not hold in general.

\begin{example}\label{ex_d1}
Let $\mathfrak{d}_1$ be the $4$-dimensional \emph{Dieudonné algebra}~\cite{LaRosaMancini1}, i.e., $\mathfrak{d}_1$ is a $2$-nilpotent algebra with basis $\lbrace e_1,e_2,e_3,z \rbrace$ and non-zero brackets
\[
[e_1,e_3]=[e_2,e_3]=-[e_3,e_1]=[e_3,e_2]=z.
\]
One has that $\mathfrak{d}_1$ is a symmetric Leibniz algebra with $\Z(\mathfrak{d}_1)=[\mathfrak{d}_1,\mathfrak{d}_1]=\langle z \rangle$. Furthermore, $x=x_1 e_1 + x_2 e_2 + x_3 e_3 \in \ZL(\mathfrak{d}_1)$ if and only if $x_2=x_3=0$. Thus, $\ZL(\mathfrak{d_1})=\langle e_1, z \rangle$.
	
We further recall from~\cite{LaRosaMancini2, ManciniBider} that $\Bider(\mathfrak{d}_1)$ can be represented by the pairs of matrices of the form
\[
(d,D)=\left( \begin{pmatrix}
a & 0 & c & 0 \\
0 & a & 0 & 0 \\
0 & 0 & b & 0 \\
a_1 & a_2 & a_3 & a+b
\end{pmatrix},\begin{pmatrix}
\frac{a+b}{2} & \frac{b-a}{2} & c+d & 0 \\
\frac{b-a}{2} & \frac{a+b}{2} & d & 0 \\
0 & 0 & b & 0 \\
b_1 & b_2 & b_3 & 0
\end{pmatrix} \right),
\]
and the inner biderivations consist of the matrices
\[
\left( \begin{pmatrix}
0 & 0 & 0 & 0 \\
0 & 0 & 0 & 0 \\
0 & 0 & 0 & 0 \\
a_1 & a_1 & a_3 & 0
\end{pmatrix},\begin{pmatrix}
0 & 0 & 0 & 0 \\
0 & 0 & 0 & 0 \\
0 & 0 & 0 & 0 \\
a_1 & -a_1 & b_3 & 0
\end{pmatrix} \right).
\]
Moreover, one may check that $\Imm(d) \subseteq \ZL(\mathfrak{d}_1)$ if and only if $a=b=0$. Thus, $\DerL(\mathfrak{d}_1)$ can be represented as the subalgebra
\[
\left\lbrace \begin{pmatrix}
0 & 0 & c & 0 \\
0 & 0 & 0 & 0 \\
0 & 0 & 0 & 0 \\
a_1 & a_2 & a_3 & 0 \\
\end{pmatrix} \; \middle| \; a_1,a_2,a_3,c \in \bF \right\rbrace
\]
of $\mathfrak{gl}(4,\bF)$, which has basis $\lbrace e_{13},e_{41},e_{42},e_{43} \rbrace$, where $e_{ij}$ denotes the elementary matrix with~$1$ in the $(i,j)$-entry and zeros elsewhere, and whose only non-zero brackets are $[e_{41},e_{13}]=-[e_{13},e_{41}]=e_{43}$. 
	
Thus, $\Inn(\mathfrak{d}_1) \subsetneq \DerL(\mathfrak{d}_1) \subsetneq \Der(\mathfrak{d}_1)$ and the Lie-holomorph $\holL(\mathfrak{d}_1)$ is the $8$-dimensional Leibniz algebra with basis
\[
\lbrace \widetilde{e}_1,\widetilde{e}_2,\widetilde{e}_3,\widetilde{z}, \ov{e_{13}},\ov{e_{41}},\ov{e_{42}},\ov{e_{43}}\rbrace,
\]
where $\widetilde{e}_i=(e_i,0)$, $\widetilde{z}=(z,0)$, $\ov{e}_{hk}=(0,e_{hk})$, and $[\holL(\mathfrak{d}_1),\holL(\mathfrak{d}_1)]=\langle \widetilde{e}_1,\widetilde{z},\ov{e_{43}} \rangle$.
\end{example}

\begin{remark}
In general, the set $\Inn(L)$ of inner derivations of $L$ is not necessarily contained in $\DerL(L)$ (see \Cref{2-dim}, where it is shown that the ideal of Lie-derivations of a non-Lie Leibniz algebra may even be trivial). However, the inclusion $\Inn(L) \subseteq \DerL(L)$ holds when $L$ is $2$-nilpotent. Indeed, in this case one has $[z,[y,x]]=[[y,x],z]=0$, for any $x,y,z \in L$, and therefore $\Imm(\ad_x)=[L,x] \subseteq \ZL(L)$.
\end{remark}

We now want to show a connection between the construction of the Lie-holomorph algebra introduced by N.~P.~Souris in~\cite{souris} and the Leibniz algebra of biderivations defined by J.-L.~Loday in~\cite{loday1993}, which uses the fact that the category $\sLeib$ is weakly action representable.

\begin{remark}
Given any Leibniz algebra $L$, there is a canonical split extension
\begin{equation}\label{can_spl}
\begin{tikzcd}
{0} & {L} & {\holL(L)} & {\DerL(L)} & {0}
\arrow[from=1-1, to=1-2]
\arrow["{i_1}", tail, from=1-2, to=1-3]
\arrow["{p_2}", twoheadrightarrow, shift left, from=1-3, to=1-4]
\arrow["{i_2}", tail, shift left, from=1-4, to=1-3]
\arrow[from=1-4, to=1-5]
\end{tikzcd}    
\end{equation}
in the category $\sLeib$, where $i_1(x)=(x,0)$, $i_2(d)=(0,d)$ and $p_2(x,d)=d$, for any $x \in L$ and for any $d \in \DerL(L)$. Thus, $\DerL(L) \cong \holL(L)/L$, and the acting morphism associated with~\eqref{can_spl} is the Leibniz algebra monomorphism
\[
\DerL(L) \rightarrowtail \Bider(L) \colon d \mapsto (-d,-d),
\]
which gives a canonical embedding of $\DerL(L)$ into $\Bider(L)$. Actually, one can prove the following.
\end{remark}

\begin{theorem}
Let $L$ be a Leibniz algebra and let $d \in \Der(L)$. The following statements are equivalent:
\begin{tfae}
\item  $d \in \DerL(L)$.
\item  $d \in \ADer(L)$.
\item $(d,d) \in \Bider(L)$.
\end{tfae}
\end{theorem}

\begin{proof}{\ }
\begin{itemize}
\item [(i)$\Rightarrow$(iii)] If $d \in \DerL(L)$, then $[d(x),y]+[y,d(x)]=0$, for any $x,y \in L$. This implies that~$d$ is also an anti-derivation. Thus, the pair $(d,d)$ is a biderivation of $L$.
\item[(iii)$\Rightarrow$(ii)] If $(d,d) \in \Bider(L)$, then trivially $d \in \ADer(L)$.
\item[(ii)$\Rightarrow$(i)] If $d \in \ADer(L)$, then
\[
d([x,y])=[d(x),y]-[d(y),x],
\]
for any $x,y \in L$. Since $d$ is also a derivation, one gets
\[
[d(y),x]=-[x,d(y)],
\]
 which implies that $d \in \DerL(L)$.
\end{itemize}    
\end{proof}

As an immediate consequence, we get that
\[
\DerL(L)=\Der(L) \cap \ADer(L) = \lbrace d \in \Der(L)\mid(d,d) \in \Bider(L) \rbrace.
\]

\section{Lie-holomorphs of low-dimensional Leibniz algebras}\label{sec:classification}

We now aim to study in detail the Lie-holomorph algebras of low-dimensional non-Lie Leibniz algebras over a field $\bF$, with $\bchar(\bF) \neq 2$. To do this, we use the classification of biderivations of such algebras given in~\cite{ManciniBider}.

Since there is no non-abelian Leibniz algebra in dimension $1$, we start with $2$-dimensional Leibniz algebras.

\subsection{Lie-holomorphs of $2$-dimensional Leibniz algebras}\label{2-dim}

Let $\dim_\bF L=2$. Then, as shown in~\cite{2-dimCuvier} by C.~Cuvier, up to isomorphism there are only two non-Lie Leibniz algebra structures on $L$:

\begin{itemize}
\item $L_A$:~nilpotent symmetric Leibniz algebra with basis $\lbrace e_1,e_2 \rbrace$ and multiplication table 
\[
[e_2,e_2]=e_1.
\]
\item $L_B$:~solvable Leibniz algebra with basis $\lbrace e_1,e_2 \rbrace$ and multiplication table
\[
[e_1,e_2]=[e_2,e_2]=e_1.
\]
\end{itemize}

We observe that $L_B$ is the only non-split non-nilpotent non-Lie Leibniz algebra with $1$-dimensional derived subalgebra (see~\cite[Corollary 2.3]{LaRosaManciniDiBartolo}).

One may check that $\ZL(L_A)=\Z(L_A)=\langle e_1 \rangle$ and $\ZL(L_B)=\Z(L_B)=0$. Furthermore, we recall from~\cite[Section 4.1]{ManciniBider} that 
\[
\Bider(L_A)=\left\lbrace \left( \begin{pmatrix}
2a & b \\
0 & a \\
\end{pmatrix},\begin{pmatrix}
0 & c \\
0 & a \\
\end{pmatrix} \right) \; \middle| \; a,b,c \in \bF \right\rbrace
\]
and
\[
\Bider(L_B)=\left\lbrace \left( \begin{pmatrix}
a & a \\
0 & 0 \\
\end{pmatrix},\begin{pmatrix}
0 & b \\
0 & 0 \\
\end{pmatrix} \right) \; \middle| \; a,b \in \bF \right\rbrace.
\]
Thus, one has $\DerL(L_A)=\Inn(L_A)=\langle e_{12} \rangle$ and $\DerL(L_B)=0$. It follows that
\[
\holL(L_A)=L_A \rtimes \Inn(L_A)=\langle \widetilde{e}_1,\widetilde{e}_2,\ov{e_{12}} \rangle,
\]
is a nilpotent Leibniz algebra with non-trivial brackets
\[
[\widetilde{e}_2,\widetilde{e}_2]=[\ov{e_{12}},\widetilde{e}_2]=-[\widetilde{e}_2,\ov{e_{12}}]=\widetilde{e}_1,
\]
while $\holL(L_B) \cong L_B$.

\subsection{Lie-holomorphs of $3$-dimensional Leibniz algebras}

The $3$-dimensional Leibniz $\mathbb{C}$-algebras and their derivations were classified in~\cite{3-dimLadra, low_dim}, while a more general classification of $3$-dimensional Leibniz algebras over a field $\bF$ with $\operatorname{char}(\bF)\neq 2$ can be found in~\cite{3-dimF}. 
Throughout this paper, we adopt the same numbering for the list of Leibniz algebras as in~\cite{3-dimF}.

Let $\dim_{\bF} L=3$ and let $\lbrace e_1,e_2,e_3 \rbrace$ be a basis of $L$ over $\bF$. The list of non-isomorphic $3$-dimensional non-Lie Leibniz $\bF$-algebras is the following.

\begin{table}[H]
	\centering
	\caption{List of non-isomorphic $3$-dimensional non-Lie Leibniz algebras.}
	\label{table:list_of_3_dim_Leib}
	\renewcommand{\arraystretch}{1.2}
	\begin{tabular}{|l|r|}
		\hline
		\textbf{Leibniz algebra}                                                                                                                                                          & \textbf{Non-trivial brackets}                                   \\
		\hline
		$L_1$                                                                                                                                                                             & $[e_1,e_3]=-2e_1,\; [e_2,e_2]=e_1,\; [e_2,e_3]=-[e_3,e_2]=-e_2$ \\
		\hline
		$L_2(\alpha),\; \alpha \neq 0$                                                                                                                                                    & $[e_1,e_3]=\alpha e_1,\; [e_2,e_3]=-[e_3,e_2]=-e_2$             \\
		\hline
		$L_3$                                                                                                                                                                             & $[e_2,e_3]=-[e_3,e_2]=-e_2,\; [e_3,e_3]=e_1$                    \\
		\hline
		$L_4$                                                                                                                                                                             & $[e_2,e_2]=[e_3,e_3]=e_1$                                       \\
		\hline
		$L_5(\alpha)$\tablefootnote{\label{note1}$\alpha$ is a non-trivial representative of the quotient group $(\bF \setminus \lbrace 0 \rbrace)/(\bF^2 \setminus \lbrace 0 \rbrace)$.} & $[e_2,e_2]=e_1,\; [e_3,e_3]=\alpha e_1$                         \\
		\hline
		$L_6(\alpha),\; \alpha \neq 0$                                                                                                                                                    & $[e_2,e_2]=[e_2,e_3]=e_1,\; [e_3,e_3]=\alpha e_1$               \\
		\hline
		$L_7$                                                                                                                                                                             & $[e_2,e_3]=e_1$                                                 \\
		\hline
		$L_8$                                                                                                                                                                             & $[e_1,e_3]=e_2,\; [e_2,e_3]=e_1$                                \\
		\hline
		$L_{9}(\alpha)$\footref{note1}                                                                                                                                                    & $[e_1,e_3]=e_2,\; [e_2,e_3]=\alpha e_1$                         \\
		\hline
		$L_{10}(\alpha),\; \alpha \neq 0$                                                                                                                                                 & $[e_1,e_3]=e_2,\; [e_2,e_3]=\alpha e_1+e_2$                     \\
		\hline
		$L_{11}$                                                                                                                                                                          & $[e_1,e_3]=e_1,\; [e_2,e_3]=e_2$                                \\
		\hline
		$L_{12}$                                                                                                                                                                          & $[e_1,e_3]=e_2,\; [e_3,e_3]=e_1$                                \\
		\hline
		$L_{13}$                                                                                                                                                                          & $[e_1,e_3]=e_1+e_2,\; [e_3,e_3]=e_1$                            \\
		\hline
	\end{tabular}
\end{table}

\begin{remark}
	The Lie-holomorph $\holL(L_A)$ is isomorphic to the $3$-dimensional Leibniz algebra $L_6\left(\tfrac{1}{4}\right)$ via the linear map 
	\[
		\holL(L_A) \to L_6\left(\frac{1}{4}\right)
	\]
	defined by
	\[
		\widetilde{e}_1 \mapsto \frac{1}{4} e_1, 
		\qquad 
		\widetilde{e}_2 \mapsto e_3, 
		\qquad 
		\overline{e_{12}} \mapsto e_2 - \frac{1}{2} e_3.
	\]
\end{remark}

In the following result we summarise the Leibniz algebras of biderivations associated with the algebras under consideration.

\begin{theorem}\cite{ManciniBider}\label{thm_bider}
The Leibniz algebras of biderivations of $3$-dimensional non-Lie Leibniz algebras over $\bF$ can be described as follows:
	\begin{itemize}[itemsep=0.5em]
		\item  $\Bider(L_1)=\left\lbrace \left(
		      \begin{pmatrix}
		      	2a & b & 0 \\
		      	0  & a & b \\
		      	0  & 0 & 0 \\
		      \end{pmatrix},
		      \begin{pmatrix}
		      	0 & -b & c \\
		      	0 & a  & b \\
		      	0 & 0  & 0 \\
		      \end{pmatrix} \right) \; \middle| \; a,b,c \in \bF \right\rbrace$.
		\item  $\Bider(L_2(\alpha))=\left\lbrace \left( \begin{pmatrix}
		      a & 0 & 0 \\
		      0 & b & c \\
		      0 & 0 & 0 \\
		\end{pmatrix},
		\begin{pmatrix}
			0 & d & e \\
			0 & b & c \\
			0 & 0 & 0 \\
		\end{pmatrix} \right) \; \middle| \; a,b,c,d,e \in \bF \right\rbrace$, $\alpha \in \bF \setminus \lbrace -1 \rbrace$.
		\item  $\Bider(L_2(-1))=\left\lbrace \left(
		      \begin{pmatrix}
		      	a & 0 & 0 \\
		      	0 & b & c \\
		      	0 & 0 & 0 \\
		      \end{pmatrix},
		      \begin{pmatrix}
		      	0 & 0 & d \\
		      	0 & b & c \\
		      	0 & 0 & 0 \\
		      \end{pmatrix} \right) \; \middle| \; a,b,c,d \in \bF \right\rbrace$.
		\item  $\Bider(L_3)=\left\lbrace \left( 
		      \begin{pmatrix}
		      	0 & 0 & b \\
		      	0 & a & c \\
		      	0 & 0 & 0 \\
		      \end{pmatrix},
		      \begin{pmatrix}
		      	0 & 0 & d \\
		      	0 & a & c \\
		      	0 & 0 & 0 \\
		      \end{pmatrix} \right) \; \middle| \; a,b,c,d \in \bF \right\rbrace$.
		\item  $\Bider(L_4)=\left\lbrace \left(
		      \begin{pmatrix}
		      	2a & b & c \\
		      	0  & a & 0 \\
		      	0  & 0 & a \\
		      \end{pmatrix},
		      \begin{pmatrix}
		      	0 & d & e \\
		      	0 & a & 0 \\
		      	0 & 0 & a \\
		      \end{pmatrix} \right) \; \middle| \; a,b,c,d,e \in \bF \right\rbrace$;
		\item  $\Bider(L_5(\alpha))=\left\lbrace \left( \begin{pmatrix}
		      2a & b & d \\
		      0 & a & -\alpha c \\
		      0 & c & a \\
		\end{pmatrix},
		\begin{pmatrix}
			0 & e & f        \\
			0 & a & \alpha c \\
			0 & c & a        \\
		\end{pmatrix} \right) \; \middle| \; a,b,c,d,e,f \in \bF \right\rbrace$.
		\item  $\Bider(L_6(\alpha))=\left\lbrace \left( \begin{pmatrix}
		      \gamma a & b & c \\
		      0 & a & \frac{a}{2} \\
		      0 & -\frac{a}{2\alpha} & (\gamma-1)a \\
		\end{pmatrix},
		\begin{pmatrix}
			0 & d                  & e           \\
			0 & a                  & \frac{a}{2} \\
			0 & -\frac{a}{2\alpha} & (\gamma-1)a \\
		\end{pmatrix} \right) \; \middle| \; a,b,c,d,e \in \bF \right\rbrace$,\\ \\ where $\gamma=\frac{4\alpha-1}{2\alpha}$. \\
		\item  $\Bider(L_7)=\left\lbrace \left(
		      \begin{pmatrix}
		      	a+b & c & d \\
		      	0   & a & 0 \\
		      	0   & 0 & b \\
		      \end{pmatrix},
		      \begin{pmatrix}
		      	0 & c & d \\
		      	0 & 0 & e \\
		      	0 & 0 & b \\
		      \end{pmatrix} \right) \; \middle| \; a,b,c,d,e \in \bF \right\rbrace$.
		\item  $\Bider(L_8)=\left\lbrace \left(
		      \begin{pmatrix}
		      	a & b & 0 \\
		      	b & a & 0 \\
		      	0 & 0 & 0 \\
		      \end{pmatrix},
		      \begin{pmatrix}
		      	0 & 0 & c \\
		      	0 & 0 & d \\
		      	0 & 0 & 0 \\
		      \end{pmatrix} \right) \; \middle| \; a,b,c,d \in \bF \right\rbrace$.
		\item  $\Bider(L_9(\alpha))=\left\lbrace \left( \begin{pmatrix}
		      a & \alpha b & 0 \\
		      b & a & 0 \\
		      0 & 0 & 0 \\
		\end{pmatrix},
		\begin{pmatrix}
			0 & 0 & c \\
			0 & 0 & d \\
			0 & 0 & 0 \\
		\end{pmatrix} \right) \; \middle| \; a,b,c,d \in \bF \right\rbrace$.
		\item  $\Bider(L_{10}(\alpha))=\left\lbrace \left( \begin{pmatrix}
		      a & \alpha b & 0 \\
		      b & a+b & 0 \\
		      0 & 0 & 0 \\
		\end{pmatrix},
		\begin{pmatrix}
			0 & 0 & c \\
			0 & 0 & d \\
			0 & 0 & 0 \\
		\end{pmatrix} \right) \; \middle| \; a,b,c,d \in \bF \right\rbrace$.
		\item  $\Bider(L_{11})=\left\lbrace \left( \begin{pmatrix}
		      a & 0 & 0 \\
		      0 & b & 0 \\
		      0 & 0 & 0 \\
		\end{pmatrix},
		\begin{pmatrix}
			0 & 0 & c \\
			0 & 0 & d \\
			0 & 0 & 0 \\
		\end{pmatrix} \right) \; \middle| \; a,b,c,d \in \bF \right\rbrace$.
		\item  $\Bider(L_{12})=\left\lbrace \left( \begin{pmatrix}
		      2a & 0 & b \\
		      b & 3a & c \\
		      0 & 0 & a \\
		\end{pmatrix},
		\begin{pmatrix}
			0 & 0 & d \\
			0 & 0 & e \\
			0 & 0 & a \\
		\end{pmatrix} \right) \; \middle| \; a,b,c,d,e \in \bF \right\rbrace$.
		\item  $\Bider(L_{13})=\left\lbrace \left( \begin{pmatrix}
		      a & 0 & a \\
		      a & 0 & b \\
		      0 & 0 & 0 \\
		\end{pmatrix},
		\begin{pmatrix}
			0 & 0 & c \\
			0 & 0 & d \\
			0 & 0 & 0 \\
		\end{pmatrix} \right)  \; \middle| \; a,b,c,d \in \bF \right\rbrace$.
	\end{itemize} \noproof
\end{theorem}

We now aim to describe the ideals $\ZL(L)$ and $\DerL(L)$ for every Leibniz algebra $L_i$ listed in \Cref{table:list_of_3_dim_Leib}.

\begin{proposition}
	For each Leibniz algebra $L_i$ listed in \Cref{table:list_of_3_dim_Leib}, the corresponding Lie-center $\ZL(L_i)$ is listed below:
	\begin{itemize}
		\item  $\ZL(L_1)=\ZL(L_8)=\ZL(L_9(\alpha))=\ZL(L_{10}(\alpha))=\ZL(L_{11})=0$.
		\item  $\ZL(L_2(\alpha))=\ZL(L_{12})=\ZL(L_{13})=\langle e_2 \rangle$.
		\item  $\ZL(L_3)=\langle e_1,e_2\rangle$.
		\item  $\ZL(L_4)=\ZL(L_5(\alpha))=\ZL(L_7)=\langle e_1 \rangle$.
		\item  $\ZL(L_6(\alpha))=\langle e_1,e_2-2e_3 \rangle$.
	\end{itemize}
\end{proposition}

\begin{proof}
For illustrative purposes, we show how we obtained the Lie-center of the Leibniz algebra $L_7$. The computations for all other cases are similar. 
	
	Let $x=x_1e_1+x_2e_2+x_3e_3 \in \ZL(L_7)$. Then, for any $y=y_1e_1+y_2e_2+y_3e_3\in L_7$, we have
	\begin{align*}
		0 & =[x,y]+[y,x]                                                     \\
		  & =x_1y_3[e_1,e_3]+x_2y_2[e_2,e_2]+x_3y_2[e_3,e_2]+x_2y_3[e_2,e_3] \\
		  & +y_1x_3[e_1,e_3]+y_2x_2[e_2,e_2]+y_3x_2[e_3,e_2]+y_2x_3[e_2,e_3] \\
		  & =(x_2 y_3 + y_2 x_3)e_1.                 
	\end{align*}
	In particular, we obtain $x_2=0$ whenever $y_2=0$, and $x_3=0$ whenever $y_3=0$. Thus, $\ZL(L_7)=\langle e_1 \rangle$.
\end{proof} 

\begin{proposition}\label{pro_der}
	For each Leibniz algebra $L_i$ listed in \Cref{table:list_of_3_dim_Leib}, the corresponding ideal of Lie-derivations $\DerL(L_i)$ is listed below:
	\begin{itemize}[itemsep=0.5em]
		\item $\DerL(L_1)=\DerL(L_8)=\DerL(L_9(\alpha))=\DerL(L_{10}(\alpha))=\DerL(L_{11})=0$.
		\item  $\DerL(L_2(\alpha))=\left\{\begin{pmatrix}
		      0 & 0 & 0 \\
		      0 & a & b \\
		      0 & 0 & 0 
		\end{pmatrix} \; \middle| \; a,b\in\bF\right\}$.
		\item  $\DerL(L_3)=\left\{\begin{pmatrix}
		      0 & 0 & b \\
		      0 & a & c \\
		      0 & 0 & 0
		\end{pmatrix} \; \middle| \; a,b,c\in\bF\right\}$.
		\item  $\DerL(L_4)=\DerL(L_5(\alpha))=\DerL(L_7)=\left\{\begin{pmatrix}
		      0 & a & b \\
		      0 & 0 & 0 \\
		      0 & 0 & 0 
		\end{pmatrix} \; \middle| \; a,b \in\bF\right\}$.
		\item  $\DerL(L_6(\alpha))=\left\{\begin{pmatrix}
		      0 & a & b \\
		      0 & 0 & 0 \\
		      0 & 0 & 0 
		\end{pmatrix} \; \middle| \; a,b \in\bF\right\}$, for any $\alpha \in \bF \setminus \lbrace \frac{1}{4}\rbrace$.
		\item  $\DerL(L_6\left(\tfrac{1}{4}\right))=\left\{\begin{pmatrix}
		      0 & b & c \\
		      0 & a & \frac{a}{2} \\
		      0 & -2a & -a 
		\end{pmatrix} \; \middle| \; a,b,c\in\bF\right\}$.
		\item  $\DerL(L_{12})=\DerL(L_{13})=\left\{\begin{pmatrix}
		      0 & 0 & 0 \\
		      0 & 0 & a \\
		      0 & 0 & 0 
		\end{pmatrix} \; \middle| \; a\in\bF\right\}$.
	\end{itemize}
\end{proposition}

\begin{proof}
For illustrative purposes, we show how we computed the Lie-derivations of the Leibniz algebra $L_7$. The computations for all other cases are similar. 
	
By \Cref{thm_bider}, we have
	\[
		\DerL(L_7)=\Der(L_7) \cap \ADer(L_7)=\left\{\begin{pmatrix}
		0 & a & b \\
		0 & 0 & 0 \\
		0 & 0 & 0 
		\end{pmatrix} \; \middle| \; a,b \in\bF\right\}=\langle e_{12},e_{13} \rangle.
	\]
	In other words, $d \in \DerL(L_7)$ if and only if $d(e_1)=0$, $d(e_2)=a e_1$ and $d(e_3)=b e_1$, for some $a,b \in \bF$.
\end{proof}

We are ready to show all Lie-holomorphs of $3$-dimensional non-Lie Leibniz algebras.

\begin{theorem}\label{theorem:3_dim_Lie_holomorphs} For each Leibniz algebra $L_i$ listed in \Cref{table:list_of_3_dim_Leib}, 
	the corresponding Lie-holomorph algebra $\holL(L_i)$ is described below. Throughout, we denote the elements $(e_j,0)$ by $\widetilde{e}_j$, for any $j=1,2,3$, and the elements $(0,e_{hk})$ by $\ov{e_{hk}}$, for any $h,k=1,2,3$.
	\begin{itemize}[itemsep=0.35em]
		\item  $\holL(L_i) \cong L_i$, for $i=1,8,11$.
		\item  $\holL(L_2(\alpha))=\langle \widetilde{e}_1,\widetilde{e}_2,\widetilde{e}_3,\ov{e_{22}},\ov{e_{23}}\rangle$, where the non-zero brackets are
		      \begin{gather*}
		      	\left[\widetilde{e}_1,\widetilde{e}_3\right]=\alpha\widetilde{e}_1,\quad
		      	\left[\widetilde{e}_3,\widetilde{e}_2\right]=-\left[\widetilde{e}_2,\widetilde{e}_3\right]=\widetilde{e}_2,\quad
		      	\left[\ov{e_{22}},\ov{e_{23}}\right]=-\left[\ov{e_{23}},\ov{e_{22}}\right]=\ov{e_{23}} \\
		      	\left[\ov{e_{22}},\widetilde{e}_2\right]=-\left[\widetilde{e}_2,\ov{e_{22}}\right]=
		      	\left[\ov{e_{23}},\widetilde{e}_3\right]=-\left[\widetilde{e}_3,\ov{e_{23}}\right]=\widetilde{e}_2.
		      \end{gather*}
		      
		\item  $\holL(L_3)=\langle \widetilde{e}_1,\widetilde{e}_2,\widetilde{e}_3,\ov{e_{13}},\ov{e_{22}},\ov{e_{23}}\rangle$, where the non-zero brackets are
		      \begin{gather*}
		      \left[\widetilde{e}_3,\widetilde{e}_2\right] = -\left[\widetilde{e}_2,\widetilde{e}_3\right] = \widetilde{e}_2, \quad
		      	\left[\widetilde{e}_3,\widetilde{e}_3\right] =\widetilde{e}_1, \quad
                \left[\ov{e_{22}},\ov{e_{23}}\right] = -\left[\ov{e_{23}},\ov{e_{22}}\right] = \ov{e_{23}},\\
		      	\left[\ov{e_{13}},\widetilde{e}_3\right] = -\left[\widetilde{e}_3,\ov{e_{13}}\right] = \widetilde{e}_1, \;
		      	\left[\ov{e_{22}},\widetilde{e}_2\right] = -\left[\widetilde{e}_2,\ov{e_{22}}\right] =\left[\ov{e_{23}},\widetilde{e}_3\right] = -\left[\widetilde{e}_3,\ov{e_{23}}\right] = \widetilde{e}_2.
		      \end{gather*}
		    
		\item  $\holL(L_4)=\langle \widetilde{e}_1,\widetilde{e}_2,\widetilde{e}_3,\ov{e_{12}},\ov{e_{13}} \rangle$, where the non-zero brackets are
		      \begin{gather*}
		      	\left[\widetilde{e}_2,\widetilde{e}_2\right] =\left[\widetilde{e}_3,\widetilde{e}_3\right]=\widetilde{e}_1,\\
		      	\left[\ov{e_{12}},\widetilde{e}_2\right] = -\left[\widetilde{e}_2,\ov{e_{12}}\right] = 
		      	\left[\ov{e_{13}},\widetilde{e}_3\right] = -\left[\widetilde{e}_3,\ov{e_{13}}\right] = \widetilde{e}_1.
		      \end{gather*}
		      
		\item  $\holL(L_5(\alpha))=\langle \widetilde{e}_1,\widetilde{e}_2,\widetilde{e}_3,\ov{e_{12}},\ov{e_{13}} \rangle$, where the non-zero brackets are
		      \begin{gather*}
		      	\left[\widetilde{e}_2,\widetilde{e}_2\right]=\widetilde{e}_1, \quad \left[\widetilde{e}_3,\widetilde{e}_3\right] = \alpha \widetilde{e}_1,\\
		      	\left[\ov{e_{12}},\widetilde{e}_2\right] = -\left[\widetilde{e}_2,\ov{e_{12}}\right] = 
		      	\left[\ov{e_{13}},\widetilde{e}_3\right] = -\left[\widetilde{e}_3,\ov{e_{13}}\right] = \widetilde{e}_1.
		      \end{gather*}
		      
		\item  $\holL(L_6(\alpha))=\langle \widetilde{e}_1,\widetilde{e}_2,\widetilde{e}_3,\ov{e_{12}},\ov{e_{13}} \rangle$, if $\alpha \in \bF \setminus \lbrace \frac{1}{4} \rbrace$, where the non-zero brackets are
		      \begin{gather*}
		      	\left[\widetilde{e}_2,\widetilde{e}_2\right]=\left[\widetilde{e}_2,\widetilde{e}_3\right]=\widetilde{e}_1, \quad \left[\widetilde{e}_3,\widetilde{e}_3\right] = \alpha \widetilde{e}_1,\\
		      	\left[\ov{e_{12}},\widetilde{e}_2\right] = -\left[\widetilde{e}_2,\ov{e_{12}}\right] = 
		      	\left[\ov{e_{13}},\widetilde{e}_3\right] = -\left[\widetilde{e}_3,\ov{e_{13}}\right] = \widetilde{e}_1.
		      \end{gather*}
		      
		\item  $\holL(L_6\left(\tfrac{1}{4}\right))=\langle \widetilde{e}_1,\widetilde{e}_2,\widetilde{e}_3,\ov{e},\ov{e_{12}},\ov{e_{13}} \rangle$, where $\ov{e}=\ov{e_{22}}+\frac{1}{2}\ov{e_{23}}-2\ov{e_{32}}-\ov{e_{33}}$ and the non-zero brackets are
		      \begin{gather*}
\left[\widetilde{e}_2,\widetilde{e}_2\right]=\left[\widetilde{e}_2,\widetilde{e}_3\right]=\widetilde{e}_1, \quad \left[\widetilde{e}_3,\widetilde{e}_3\right] = \tfrac{1}{4} \widetilde{e}_1,\\ \left[\ov{e_{12}},\ov{e}\right]=-\left[\ov{e},\ov{e_{12}}\right]=\ov{e_{12}}+\frac{1}{2}\ov{e_{13}}, \quad \left[\ov{e},\ov{e_{13}}\right]=-\left[\ov{e_{13}},\ov{e}\right]=2\ov{e_{12}}+\ov{e_{13}},\\
		      	\left[\ov{e},\widetilde{e}_2\right]=-\left[\widetilde{e}_2,\ov{e}\right]=2\left[\ov{e},\widetilde{e}_3\right]=-\left[\widetilde{e}_3,\ov{e}\right]=\widetilde{e}_2-2\widetilde{e}_3,\\
		      	\left[\ov{e_{12}},\widetilde{e}_2\right] = -\left[\widetilde{e}_2,\ov{e_{12}}\right] = 
		      	\left[\ov{e_{13}},\widetilde{e}_3\right] = -\left[\widetilde{e}_3,\ov{e_{13}}\right] = \widetilde{e}_1.
		      \end{gather*}
		      
		\item  $\holL(L_7)=\langle \widetilde{e}_1,\widetilde{e}_2,\widetilde{e}_3,\ov{e_{12}},\ov{e_{13}} \rangle$, where the non-zero brackets are
		      \begin{gather*}
		      	\left[\widetilde{e}_2,\widetilde{e}_3\right]=
		      	\left[\ov{e_{12}},\widetilde{e}_2\right] = -\left[\widetilde{e}_2,\ov{e_{12}}\right] = 
		      	\left[\ov{e_{13}},\widetilde{e}_3\right] = -\left[\widetilde{e}_3,\ov{e_{13}}\right] = \widetilde{e}_1.
		      \end{gather*}
		      
		\item  $\holL(L_i(\alpha)) \cong L_i(\alpha)$, for $i=9,10$.
		      
		\item  $\holL(L_{12})=\langle \widetilde{e}_1,\widetilde{e}_2,\widetilde{e}_3,\ov{e_{23}} \rangle$, where the non-zero brackets are
		      \begin{gather*}
		      	\left[\widetilde{e}_1,\widetilde{e}_3\right]=\widetilde{e}_2, \quad \left[\widetilde{e}_3,\widetilde{e}_3\right]=\widetilde{e}_1, \quad \left[\ov{e_{23}},\widetilde{e}_3\right] = -\left[\widetilde{e}_3,\ov{e_{23}}\right] = \widetilde{e}_2.
		      \end{gather*}
		      
		\item  $\holL(L_{13})=\langle \widetilde{e}_1,\widetilde{e}_2,\widetilde{e}_3,\ov{e_{23}} \rangle$, where the non-zero brackets are
		      \begin{gather*}
		      	\left[\widetilde{e}_1,\widetilde{e}_3\right]=\widetilde{e}_1+\widetilde{e}_2, \quad \left[\widetilde{e}_3,\widetilde{e}_3\right]=\widetilde{e}_1, \quad \left[\ov{e_{23}},\widetilde{e}_3\right] = -\left[\widetilde{e}_3,\ov{e_{23}}\right] = \widetilde{e}_2.
		      \end{gather*}
	\end{itemize}
\end{theorem}

\begin{proof}
	For illustrative purposes, we show how we obtained the Lie-holomorph of the Leibniz algebra $L_7$. The computations for all other cases are similar. 
	
	By \Cref{pro_der}, we have that
	$\holL(L_7)=L_7 \rtimes \DerL(L_7)=L_7 \rtimes \langle e_{12}, e_{13} \rangle$ is the $5$-dimensional Leibniz algebra with basis $\lbrace \widetilde{e}_1,\widetilde{e}_2,\widetilde{e}_3,\ov{e_{12}}, \ov{e_{13}} \rbrace$ and non-trivial commutators
	\begin{align*}
		  & \left[\widetilde{e}_2,\widetilde{e}_3\right]=\left[(e_2,0),(e_3,0)\right]=(e_1,0)=\widetilde{e}_1,                         \\
		  & \left[\ov{e_{12}},\widetilde{e}_2\right] = \left[(0,e_{12}),(e_2,0)\right]=(e_{12} \cdot e_2,0)=(e_1,0)=\widetilde{e}_1,   \\
		  & \left[\widetilde{e}_2,\ov{e_{12}}\right] = \left[(e_2,0),(0,e_{12})\right]=(-e_{12}\cdot e_2,0)=(-e_1,0)=-\widetilde{e}_1, \\
		  & \left[\ov{e_{13}},\widetilde{e}_3\right] = \left[(0,e_{13}),(e_3,0)\right]=(e_{13} \cdot e_3,0)=(e_1,0)=\widetilde{e}_1,   \\
		  & \left[\widetilde{e}_3,\ov{e_{13}}\right] = \left[(e_3,0),(0,e_{13})\right]=(-e_{13}\cdot e_3,0)=(-e_1,0)=-\widetilde{e}_1. 
	\end{align*}
\end{proof}

We observe that the dimension of the Lie-holomorph algebra of a $3$-dimensional Leibniz algebra lies between $3$ and $6$. Furthermore:
\begin{itemize}[itemsep=0.35em]
	\item The $4$-dimensional Leibniz algebras $\holL(L_{12})$ and $\holL(L_{13})$ are not isomorphic. Indeed, the former is $3$-nilpotent, since
	      \[
	      	\holL(L_{12})^{(1)}=\langle \widetilde{e}_1,\widetilde{e}_2 \rangle, \quad \holL(L_{12})^{(2)}=\langle \widetilde{e}_2 \rangle, \quad \holL(L_{12})^{(3)}=0,
	      \]
	      whereas the latter is not nilpotent, because
	      \[
	      	\holL(L_{13})^{(1)}=\langle \widetilde{e}_1,\widetilde{e}_2 \rangle, \quad \holL(L_{13})^{(k)}=\langle \widetilde{e}_1+\widetilde{e}_2 \rangle, \quad \text{for any } k \geq 2.
	      \]
	      We observe that $\holL(L_{12})$ is isomorphic to the Leibniz algebra $\mathfrak{R}_5$ of~\cite[Theorem 3.2]{4-dim-nil}, which has basis $\lbrace x_1,x_2,x_3,x_4 \rbrace$ and non-trivial commutators
	      \[
	      	[x_1,x_1]=x_3, \quad [x_1,x_2]=[x_3,x_1]=x_4.
	      \]
	      An explicit isomorphism is given by the linear map $\varphi \colon \mathfrak{R}_5 \to \holL(L_{12})$ defined by
	      \[
	      	\varphi(x_1)=\widetilde{e}_3, \quad \varphi(x_2)=\widetilde{e}_1-\ov{e_{23}}, \quad \varphi(x_3)=\widetilde{e}_1, \quad \varphi(x_4)=\widetilde{e}_2.
	      \]
Furthermore, since passing to the opposite algebra reverses the Leibniz identity (turning right Leibniz algebras into left Leibniz algebras), it is natural to compare $\holL(L_{12})^{\op}$ with the classification of $4$-dimensional nilpotent left Leibniz algebras given in~\cite{4-dim-nil}. One may verify that $\holL(L_{12})^{\op}$ is indeed isomorphic to the left Leibniz algebra $\mathcal{A}_{24}$ of~\cite[Theorem 2.6]{4-dim-nil}, since $\mathcal{A}_{24}^{\op} \cong \mathfrak{R}_5$.
	      
Finally, the Lie-holomorph $\holL(L_{13})$ is isomorphic to the Leibniz algebra $\mathfrak{L}_{39}$ of~\cite[Proposition 3.11]{4-dim-solv}, which has basis $\lbrace f_1,f_2,f_3,f_4 \rbrace$ and non-zero brackets
	      \[
	      	[f_1,f_4]=-[f_4,f_1]=f_2, \quad [f_3,f_4]=f_3, \quad [f_4,f_4]=f_2.
	      \]
	      Indeed, one may check that the linear map $\psi \colon \mathfrak{L}_{39} \to \holL(L_{13})$ defined by
	      \[
	      	\psi(f_1)=\widetilde{e}_2-\ov{e_{23}}, \quad \psi(f_2)=-\widetilde{e}_2, \quad \psi(f_3)=\widetilde{e}_1+\widetilde{e}_2, \quad \psi(f_4)=-\widetilde{e}_1+\widetilde{e}_3,
	      \]
	      is a Leibniz algebra isomorphism.
	\item The algebra $\holL(L_2(\alpha))$ is not isomorphic to any of the other $5$-dimensional Lie-holomorphs listed in \Cref{theorem:3_dim_Lie_holomorphs}, since it is the only one among them with a $3$-dimensional derived subalgebra.
	      
	\item For any $\alpha\in\bF \setminus \lbrace \tfrac{1}{4} \rbrace$, the Leibniz algebra $\holL(L_6(\alpha))$ is not isomorphic to any of the other $5$-dimensional Lie-holomorphs listed in \Cref{theorem:3_dim_Lie_holomorphs}, since it is the only one among them with a $4$-dimensional Lie-center.
	      
	\item The Lie-holomorphs $\holL(L_5(-1))$ and $\holL(L_7)$ are isomorphic. Indeed, the change of basis
	      \begin{gather*}
	      	\widetilde{e}_1 \mapsto \widetilde{e}_1, \quad \widetilde{e}_2 \mapsto \widetilde{e}_2+\widetilde{e}_3+\ov{e_{12}}, \quad \widetilde{e}_3 \mapsto \widetilde{e}_2-\widetilde{e}_3, \\
	      	\ov{e_{12}} \mapsto \frac{1}{2}(\ov{e_{12}}+\ov{e_{13}}), \quad \ov{e_{13}} \mapsto \frac{1}{2}(\ov{e_{12}}-\ov{e_{13}})  
	      \end{gather*}
	      defines a Leibniz algebra isomorphism $\holL(L_5(-1)) \to \holL(L_7)$.
	      
	\item The algebras $\holL(L_7)$ and $\holL(L_4)$ are not isomorphic over a general field $\bF$. Indeed, suppose there is an isomorphism $\phi \colon \holL(L_4) \to \holL(L_7)$. Then
	      \[
	      	\phi(\widetilde{e}_1)=k\widetilde{e}_1,
	      \]
	      for some $k \neq 0$, since $\left[\holL(L_7),\holL(L_7) \right]=\left[\holL(L_4),\holL(L_4) \right]=\langle \widetilde{e}_1 \rangle$, and
	      \[
	      	\phi(\widetilde{e}_2)=\alpha_1 \widetilde{e}_1+\alpha_2\widetilde{e}_2+\alpha_3\widetilde{e}_3+\alpha_4\ov{e_{12}}+\alpha_5 \ov{e_{13}},
	      \]
	      \[
	      	\phi(\widetilde{e}_3)=\beta_1 \widetilde{e}_1+\beta_2\widetilde{e}_2+\beta_3\widetilde{e}_3+\beta_4\ov{e_{12}}+\beta_5 \ov{e_{13}},
	      \]
	      for some $\alpha_i,\beta_j \in \bF$. Thus, one has
	      \begin{align*}
	      	k \widetilde{e}_1=\phi([\widetilde{e}_2,\widetilde{e}_2])= & [\phi(\widetilde{e}_2),\phi(\widetilde{e}_2)]=\alpha_2 \alpha_3 \widetilde{e}_1,                                                                            \\
	      	k \widetilde{e}_1=\phi([\widetilde{e}_3,\widetilde{e}_3])= & [\phi(\widetilde{e}_3),\phi(\widetilde{e}_3)]=\beta_2 \beta_3 \widetilde{e}_1,                                                                              \\
	      	0=\phi([\widetilde{e}_2,\widetilde{e}_3])=                 & [\phi(\widetilde{e}_2),\phi(\widetilde{e}_3)]=(\alpha_2 \beta_3 - \alpha_2\beta_4 - \alpha_3 \beta_5 +\alpha_4 \beta_2 + \alpha_5 \beta_3) \widetilde{e}_1, \\
	      	0=\phi([\widetilde{e}_3,\widetilde{e}_2])=                 & [\phi(\widetilde{e}_3),\phi(\widetilde{e}_2)]=(\beta_2 \alpha_3 - \beta_2\alpha_4 - \beta_3 \alpha_5 +\beta_4 \alpha_2 + \beta_5 \alpha_3) \widetilde{e}_1. 
	      \end{align*}
	      It follows that $k=\alpha_2\alpha_3=\beta_2\beta_3$ and $\alpha_2 \beta_3+\alpha_3 \beta_2=0$. Hence, $\alpha_i \neq 0 \neq \beta_i$ for any~$i=2,3$, and
	      \[
	      	k\left(\dfrac{\alpha_2}{\beta_2}+\dfrac{\beta_2}{\alpha_2}\right)=0.
	      \]
	      Hence, we have $\alpha_2^2+\beta_2^2=0$, which has a non-trivial solution if and only if $-1$ is a square in $\bF$. In this case, we recall from~\cite{3-dimF} that $L_5(-1)$ is isomorphic to $L_4$, which implies
	      \[
	      	\holL(L_4) \cong \holL(L_5(-1)) \cong \holL(L_7).
	      \]
	      Therefore, we conclude that $\holL(L_7)$ is isomorphic to $\holL(L_4)$ if and only if $x^2+1=0$ has a solution in $\bF$. 
	      
	\item The $6$-dimensional Leibniz algebras $\holL(L_3)$ and $\holL(L_6\left(\tfrac{1}{4}\right))$ are not isomorphic, since
    \begin{equation*}
\dim_{\bF}\Leib\left(\holL(L_3)\right)=1\neq 2= \dim_{\bF}\Leib\left(\holL\left(L_6\left(\tfrac{1}{4}\right)\right)\right).
    \end{equation*}
    Indeed, it is easy to see that $\Leib(\holL(L_3))=\langle \widetilde{e}_1 \rangle$, while
    \[
    \widetilde{e}_1, \widetilde{e}_2-2 \widetilde{e}_3 \in \Leib\left(\holL\left(L_6\left(\tfrac{1}{4}\right)\right)\right)
    \]
    since
    \[
[\widetilde{e}_2,\widetilde{e}_2]=\widetilde{e}_1 \quad \text{ and } \quad [\widetilde{e}_3+\ov{e},\widetilde{e}_3+\ov{e}]=\frac{1}{4}\widetilde{e}_1-\frac{1}{2}(\widetilde{e}_2-2\widetilde{e}_3). \]
\end{itemize}

We summarise the results of \Cref{sec:classification} in the following theorem.

\begin{theorem}
Let $L$ be a non-Lie Leibniz algebra with $\dim_\bF L \leq 3$. Then, the Lie-holomorph $\holL(L)$ is isomorphic to one of the Leibniz algebras listed in~\Cref{tab:holomorphs}, which are pairwise non-isomorphic.
\noproof
\end{theorem}

\renewcommand{\arraystretch}{1.4}

\begin{table}[H]
	\centering
	\caption{Classification of Lie-holomorphs of low-dimensional Leibniz algebras}
	\label{tab:holomorphs}
	\begin{tabular}{|c|c|}
		\hline
		$\dim_\bF \holL(L)$ & \textbf{Lie-holomorph}                                                                                   \\ \hline
		2                            & $L_B$                                                                                                    \\ \hline
		3                            & $L_1$, $L_6\left(\tfrac{1}{4}\right)$, $L_8$, $L_9(\alpha)$, $L_{10}(\alpha)$, $L_{11}$                  \\ \hline
		4                            & $\mathfrak{R}_5$, $\mathfrak{L}_{39}$                                                                    \\ \hline
		5                            & $\holL(L_2(\alpha))$, $\holL(L_4)$, $\holL(L_5(\alpha))$, $\holL(L_6(\beta))$, $\beta \neq \tfrac{1}{4}$ \\ \hline
		6                            & $\holL(L_3)$, $\holL\left(L_6\left(\frac{1}{4}\right)\right)$                                            \\ \hline
	\end{tabular}
\end{table}


\section*{Acknowledgment}

The authors are supported by the University of Palermo, by the ``National Group for Algebraic and Geometric Structures and their Applications'' (GNSAGA -- INdAM), and by the SDF Sustainability Decision Framework Research Project -- MISE decree of 31/12/2021 (MIMIT Dipartimento per le politiche per le imprese -- Direzione generale per gli incentivi alle imprese) -- CUP:~B79J23000530005, COR:~14019279, Lead Partner:~TD Group Italia Srl, Partner:~University of Palermo. The second author is also a Postdoctoral Researcher of the Fonds de la Recherche Scientifique--FNRS.

%
%
%
%

\printbibliography
\end{document}